\documentclass[12pt,reqno]{amsart}

\usepackage{amssymb}

\newtheorem{theorem}{Theorem}[section]
\newtheorem{proposition}[theorem]{Proposition}
\newtheorem{lemma}[theorem]{Lemma}

\theoremstyle{remark}
\newtheorem{definition}[theorem]{Definition}
\newtheorem{remark}[theorem]{Remark}

\numberwithin{equation}{section}

\begin{document}

\title[Direct images of parabolic bundles and connections]{On the direct images of 
parabolic vector bundles and parabolic connections}

\author[Indranil Biswas]{Indranil Biswas}

\address{School of Mathematics, Tata Institute of Fundamental Research,
Homi Bhabha Road, Mumbai 400005, India\\ and
Mathematics Department, EISTI-University Paris-Seine, Avenue du parc, 95000,
Cergy-Pontoise, France}

\email{indranil@math.tifr.res.in}

\author[Francois-Xavier Machu]{Francois-Xavier Machu}

\address{Mathematics Department, EISTI-University Paris-Seine, Avenue du parc, 95000, 
Cergy-Pontoise, France}

\email{fmu@eisti.eu}

\subjclass[2010]{14E20, 14J60, 53B15}

\keywords{Parabolic bundle; parabolic connection; ramified torus bundle;
parabolic direct image.}

\date{}

\begin{abstract}
Let $\varphi\, :\, Y\, \longrightarrow\, X$ be a finite surjective morphism between 
smooth complex projective curves, where $X$ is irreducible but $Y$ need not be so. Let 
$V_*$ be a parabolic vector bundle on $Y$. We construct a parabolic structure on the
direct image
$\varphi_* V$ on $X$, where $V$ is the vector bundle underlying $V_*$. The parabolic vector 
bundle $\varphi_* V_*$ on $X$ obtained this way has a ramified torus sub-bundle; it is a 
torus bundle of $\text{Ad}(\varphi_* V)$ outside the parabolic divisor for $\varphi_* V_*$ that satisfies 
certain conditions at the parabolic points. Conversely, given a parabolic vector bundle 
$E_*$ on $X$, and a ramified torus sub-bundle $\mathcal T$ for it, we construct a ramified 
covering $Z$ of $X$ and a parabolic vector bundle $W_*$ on $Z$, such that
the parabolic bundle $E_*$ is the 
direct image of $W_*$. A connection on $V_*$ produces a connection on $\varphi_* V_*$. 
The ramified torus sub-bundle for $\varphi_* V_*$ is preserved by the logarithmic 
connection on $\text{End}(\varphi_* V)$ induced by this connection on $\varphi_* V_*$. 
If the parabolic vector bundle $E_*$ on $X$ is equipped with a connection $D$ such that 
the connection on the endomorphism bundle induced by it preserves the ramified torus 
sub-bundle $\mathcal T$, then we prove that the corresponding parabolic vector bundle
$W_*$ on $Z$ has a connection that produces the connection $D$ on the direct image $E_*$.
\end{abstract}

\maketitle

\tableofcontents

\section{Introduction}\label{sec0}

Let $X$ be a smooth complex projective curve and $S$ an effective reduced divisor
on $X$. A parabolic bundle on $X$ with parabolic divisor $S$ is a vector bundle
on $X$ equipped with a filtration of subspaces of the fiber over
each point of $S$ together with a system of weights for these filtrations. Parabolic
bundles were introduced by Mehta and Seshadri \cite{MS}; parabolic vector bundles
on higher dimensional varieties were introduced by Maruyama and Yokogawa \cite{MY}.

Let $\varphi\, :\, Y\, \longrightarrow\, X$ be a ramified covering. Take a parabolic
vector bundle $V_*$ on $Y$ with underlying vector bundle $V$. We construct a
parabolic structure on the direct image $\varphi_*V$ over $X$. This is done using the
parabolic structure of $V_*$ and the ramifications of the covering map $\varphi$;
the details are in Section \ref{se2}. This construction is compatible with
other structures on parabolic bundles. For example, a connection on $V_*$ produces a connection
on the direct image $\varphi_*V_*$ (see Theorem \ref{thm1}).

Let $E_*$ be a parabolic vector bundle on $X$ with parabolic divisor $S$ and rational
parabolic weights. Denote the vector bundle underlying $E_*$ by $E$. Let
$$\text{Ad}(E)\, \subset\, \text{End}(E)\,=\, E\otimes E^*$$ be the locus of invertible endomorphisms.
A ramified torus sub-bundle for $E_*$ is a torus sub-bundle (fibers are tori)
$$
{\mathcal T}\, \subset\, \text{Ad}(E)\vert_{X\setminus S}
$$
satisfying a certain condition on $S$ (see Definition \ref{drts}). We show that
the above mentioned direct image $\varphi_*V_*$ has a natural ramified torus sub-bundle
(Proposition \ref{prop2}). If $D$ is a connection on $V_*$, then the connection
on $\varphi_*V_*$ given by $D$ is compatible with this ramified torus sub-bundle
for $\varphi_*V_*$ (Lemma \ref{lem2}).

In the reverse direction, take a parabolic vector bundle $E_*$ on $X$ with
rational parabolic weights. Take a ramified torus sub-bundle $\mathcal T$ for
$E_*$. Then there is a ramified covering $$\phi\, :\, Z\, \longrightarrow\, X$$
and a parabolic vector bundle $W_*$ on $Z$, such that $\phi_*W_*\,=\, E_*$, and the
ramified torus sub-bundle for the direct image $\phi_*W_*$ coincides with $\mathcal T$
(see Proposition \ref{prop3}). We prove the following in Theorem
\ref{thm2}:

\begin{theorem}\label{thm01}
Let $E_*$ be a parabolic vector bundle on a connected smooth complex projective curve $X$
with parabolic divisor $S$ and rational parabolic weights. Then there is a natural
equivalence between the following two classes:
\begin{enumerate}
\item Triples of the form $(Y,\, \varphi,\, V_*)$, where $\varphi\, :\, Y\, \longrightarrow
\, X$ is a ramified covering map, and $V_*$ is a parabolic vector bundle on $Y$, such that
$\varphi_*V_*\, =\, E_*$.

\item Ramified torus bundles for $E_*$.
\end{enumerate}
\end{theorem}

When the torus is maximal, which means that the parabolic bundle $V_*$ in
Theorem \ref{thm01} is a line bundle, then Theorem \ref{thm01} was proved in 
\cite{AB}.

As before, $E_*$ is a parabolic vector bundle on $X$ with rational parabolic 
weights, and $\mathcal T$ is a ramified torus sub-bundle for $E_*$. Now take a 
connection on $D$ that preserves $\mathcal T$. We prove that $D$ induces a
connection $D'$ on the corresponding parabolic vector bundle $W_*$ over the
covering $Z$ such that the connection on $\phi_*W_*\,=\, E_*$ given by $D'$
coincides with $D$ (see Proposition \ref{prop4}). We prove the following
in Theorem \ref{thm3}:

\begin{theorem}\label{thm02}
The equivalence in Theorem \ref{thm01} takes a connection on the parabolic
vector bundle $V_*$ on $Y$ to a connection on $E_*$ that preserves the
ramified torus sub-bundle for $E_*$ corresponding to $(Y,\, \varphi,\, V_*)$. Conversely,
a connection on $E_*$ preserving a ramified torus sub-bundle $\mathcal T$ for $E_*$
is taken to a connection on the parabolic vector bundle $V_*$ on $Y$, where
$(Y,\, \varphi,\, V_*)$ corresponds to $\mathcal T$.
\end{theorem}

When there is no parabolic structure and the coverings are unramified, then Theorem 
\ref{thm01} and Theorem \ref{thm02} were proved in \cite{BC}. Direct images of vector 
bundles under finite maps of curves are also studied in \cite{DP}.

\section{Parabolic vector bundles and connections}\label{sec1}

\subsection{Parabolic vector bundles}

Let $X$ be an irreducible smooth complex projective curve. Take an algebraic vector 
bundle $E$ over $X$ of rank $r$. The fiber of $E$ over any closed point $z\, \in\, X$ 
will be denoted by $E_z$.

A \textit{quasi-parabolic} structure on $E$ consists of the following data:
\begin{itemize}
\item a finite set of distinct closed points $S\,=\, \{x_1,\, \cdots,\, x_n\}\,\subset\,
X$, and

\item for each $x_i\, \in\, S$, a filtration of subspaces
$$
E_{x_i}\,=\, F^i_1 \, 
\supsetneq\, F^i_2\, \supsetneq\, \cdots \, \supsetneq\, F^i_{\ell_i}\, \supsetneq\, 
F^i_{\ell_i+1}\,=\, 0\, .
$$
\end{itemize}
The above subset $S$ is called the \textit{parabolic divisor}. A \textit{quasi-parabolic 
bundle} is a vector bundle equipped with a quasi-parabolic structure. A 
\textit{parabolic vector bundle} is a quasi-parabolic vector bundle $(E,\, S,\, 
\{\{F^i_j\}^{j=\ell_i}_{j=1}\}_{i=1}^n)$ as above together with real numbers 
$\lambda^i_j$, $1\,\leq\, j\,\leq\, \ell_i$, $1\,\leq\, i\, \leq\, n$, such that
$$
0\, \leq \, \lambda^i_1\, <\, \lambda^i_2\, 
<\, \cdots\, < \, \lambda^i_{\ell_i-1}\, < \, \lambda^i_{\ell_i} \, <\, 
\lambda^i_{\ell_i+1}\,=\, 1\, .
$$
These numbers $\lambda^i_j$ are called the \textit{parabolic weights}. For notational 
convenience, a parabolic vector bundle $(E,\, S,\, \{\{(F^i_j,\, \lambda^i_j) 
\}^{j=\ell_i}_{j=1}\}_{i=1}^n)$ is abbreviated as $E_*$. See \cite{MS}, \cite{MY} for 
more on parabolic vector bundles.

For any $x_i\, \in\, S$, and any real number $0\, \leq\, c\, \leq\, 1$, define 
\begin{equation}\label{f1}
E^c_{x_i}\,:=\, F^i_{j(c)}\, \subset\, E_{x_i} \, , 
\end{equation}
where $1\, \leq\, j(c)\, \leq\, \ell_i+1$ is the smallest integer such 
that $c\, \leq \, \lambda^i_{j(c)}$. Evidently, the parabolic structure of $E_*$ at 
the parabolic point $x_i$ is uniquely determined by the weighted subspaces 
$\{E^c_{x_i}\}_{0\leq c\leq 1}$ of $E_{x_i}$. We note that in \cite[p.~80]{MY}, the 
parabolic structure is in fact defined this way.

The parabolic degree of $E_*$, which is denoted by $\text{par-deg}(E_*)$, is defined 
to be \begin{equation}\label{pdeg}
\text{par-deg}(E_*)\,:=\, 
\text{degree}(E)+\sum_{i=1}^n\sum_{j=1}^{\ell_i}\lambda^i_j \cdot \dim 
F^i_j/F^i_{j+1}
\end{equation}
\cite[p.~214, Definition~1.11]{MS}, \cite[p.~78]{MY}.

The notions tensor product, dualization and homomorphism bundles for usual vector 
bundles extend to the context of parabolic vector bundles \cite{Yo}, \cite{Bi2}. More 
precisely, for any parabolic vector bundles $E_*$ and $E'_*$, there are the parabolic 
tensor product bundle $E_*\otimes E'_*$, parabolic dual bundle $E^*_*$ and parabolic 
homomorphism bundle $\text{Hom}(E_*,\, E'_*) \,=\, E'_*\otimes E^*_*$; there are all 
parabolic vector bundles (see \cite{Yo}, \cite{Bi2} for their construction). However, we 
will recall below the description of the vector bundle underlying the parabolic 
bundle $\text{End}(E_*)\,=\, \text{Hom}(E_*,\, E_*) \,=\, E_*\otimes E^*_*$.

As before, let $E$ be the vector bundle underlying the parabolic bundle $E_*$. Consider the 
vector bundle $$\text{End}(E)\,=\, \text{Hom}(E,\, E) \,=\, E\otimes E^*\, .$$ Let
\begin{equation}\label{e1}
\text{End}^p(E)\, \subset\, \text{End}(E)
\end{equation}
be the sub-sheaf such that for any section $$s\, \in\, \Gamma (U,\,
\text{End}^p(E)) \, \subset\, \Gamma (U,\, \text{End}(E))$$ defined over an open subset
$U\, \subset\, X$, we have
$$
s(x_i)(F^i_j)\, \subset\, F^i_j
$$
for all $x_i\, \in S\bigcap U$ and $1\,\leq\, j\,\leq\, \ell_i$, where $S$ as before denotes the
parabolic divisor for $E_*$. Note that the inclusion of $\text{End}^p(E)$ in $\text{End}(E)$
is an isomorphism over the complement $X\setminus S$.

The vector bundle underlying the parabolic bundle $\text{End}(E_*)$ is
$\text{End}^p(E)$ constructed in \eqref{e1}.

Let $Y$ be a smooth complex projective curve which need not be connected. By a 
vector bundle on $Y$ we will mean a vector bundle on every connected component of 
$Y$. Note that we do not assume that they all have the same rank independent of the 
components. Similarly, by a parabolic bundle on $Y$ we will mean a parabolic vector 
bundle on every connected component of $Y$; their rank is allowed to be different 
on different connected components.

\subsection{Connections on a parabolic vector bundle}

The cotangent line bundle of $X$ will be denoted by $K_X$.

As before, let $E_*$ be a parabolic vector bundle with parabolic divisor $S$ such 
that the underlying vector bundle is $E$. The line bundle $K_X\otimes {\mathcal 
O}_X(S)$ will be denoted by $K_X(S)$ for notational
convenience.

We note that for any point $y\, \in\, S$, the 
fiber $K_X(S)_y$ is identified with $\mathbb C$ by sending any meromorphic $1$-form defined around 
$y$ to its residue at $y$. More precisely, for any holomorphic coordinate function 
$z$ on $X$ defined around the point $y$, consider the homomorphism
\begin{equation}\label{e2}
R_y\, :\, K_X(S)_y\, \longrightarrow\, {\mathbb C}\, ,\ \ c\cdot \frac{dz}{z} \, \longmapsto\, c\, .
\end{equation}
This homomorphism is in fact independent of the choice of the above coordinate function $z$.

A \textit{logarithmic connection} on $E$ singular over $S$ is an algebraic 
differential operator of order one
$$
D\, :\, E\, \longrightarrow\, E\otimes K_X(S)
$$
such that $D(fs) \,=\, f D(s) + s\otimes df$ for all locally defined algebraic function $f$ and
locally defined algebraic section $s$ of $E$.

For a logarithmic connection $D$ on $E$ singular over $S$, and a point $y\, \in\, S$, consider
the composition
$$
E\, \stackrel{D}{\longrightarrow}\, E\otimes K_X(S)\, \stackrel{\text{Id}_E\otimes
R_y}{\longrightarrow}\, E_y\otimes{\mathbb C}\,=\, E_y\, ,
$$
where $R_y$ is the homomorphism in \eqref{e2}. This composition homomorphism evidently vanishes
on the sub-sheaf $E\otimes {\mathcal O}_X(-y)\, \subset\, E$, and hence it produces a homomorphism
\begin{equation}\label{e3}
\text{Res}(D,y)\, :\, E/(E\otimes {\mathcal O}_X(-y))\,=\, E_y\, \longrightarrow\, E_y\, .
\end{equation}
The homomorphism $\text{Res}(D,y)$ in \eqref{e3} is called the \textit{residue} of the
connection $D$ at the point $y$; see \cite[p.~53]{De}.

\begin{definition}\label{dlc}
A \textit{connection} on $E_*$ is a logarithmic connection $D$ on $E$, singular over
$S$, such that
\begin{enumerate}
\item $\text{Res}(D,x_i)(F^i_j)\, \subset\, F^i_j$ for all $1\,\leq\, j\,\leq\, \ell_i$,
$1\,\leq\, i\, \leq\, n$, and

\item the endomorphism of $F^i_j/F^i_{j+1}$ induced by $\text{Res}(D,x_i)$ coincides with
multiplication by the parabolic weight $\lambda^i_j$.
\end{enumerate}
(See \cite[Section~2.2]{BL}.)
\end{definition}

If $D$ is a logarithmic connection on $E$ singular over $S$, then
$$
\text{degree}(E)\,=\, - \sum_{i=1}^n\text{trace}(\text{Res}(D,x_i))
$$
\cite[p.~16, Theorem~3]{Oh}. From this, and the definition of parabolic degree
reproduced in \eqref{pdeg}, it follows immediately that if a parabolic bundle $E_*$ 
admits a connection, then
$$
\text{par-deg}(E_*)\,=\,0\, .
$$
Conversely, an indecomposable parabolic vector bundle of parabolic degree zero
admits a connection \cite[Proposition 4.1]{BL}.

A connection on $E_*$ induces a connection on the parabolic dual $E^*_*$. Also,
if $E_*$ and $E'_*$ are equipped with connections, then there are induced connections
on $E_*\otimes E'_*$ and $\text{Hom}(E_*,\, E'_*)$. In particular, a connection
on $E$ induces a connection on the parabolic endomorphism bundle $\text{End}(E_*)$.

Let $Y$ be a smooth complex projective curve which need not be connected, and let $E_*$
be a parabolic vector bundle on $Y$. Then a connection on $E_*$ is defined to be a
connection on the parabolic vector bundle over each connected component of $Y$.

\section{Direct image of a parabolic vector bundle}\label{se2}

Let $Y$ be a smooth complex projective curve which need not be connected.
As before, the curve $X$ is irreducible. Let
\begin{equation}\label{vp}
\varphi\, :\, Y\, \longrightarrow\, X
\end{equation}
be a finite morphism.

Take a parabolic vector bundle $V_*$ on $Y$ of positive rank. The
vector bundle underlying $V_*$ will be denoted by $V$. Note that the
direct image $\varphi_*V\,\longrightarrow\, X$ is a vector bundle, because
$\varphi$ is a finite morphism so the first direct image vanishes. We will
construct a parabolic structure on the vector bundle $\varphi_*V$.

Since $\varphi$ is a finite morphism, there are only finitely points of $Y$ where $\varphi$ 
fails to be smooth. Let
\begin{equation}\label{dr}
R\, \subset\, Y
\end{equation}
be the finite subset where the map $\varphi$ is ramified. Let $P\, \subset\, Y$ be the parabolic 
divisor for the parabolic vector bundle $V_*$. The parabolic divisor for the parabolic structure 
on $\varphi_*V$ is the image $\varphi(R\bigcup P)$. Take a point $x\, \in\, \varphi(R\bigcup 
P)\setminus \varphi(R)$ in the complement of $\varphi(R)\,\subset\, \varphi(R\bigcup P)$. Then
we have
$$
(\varphi_* V)_x\,=\, \bigoplus_{y\in \varphi^{-1}(x)} V_y\, .
$$
Since $x\, \in\, \varphi(R\bigcup P)\setminus \varphi(R)$, some points of $\varphi^{-1}(x)$ are
parabolic points for $V_*$. If $y\,\in\, \varphi^{-1}(x)$ is not a parabolic point of $V_*$,
equip $V_y$ with the trivial quasi-parabolic filtration
$$
V_{y}\,=\, V^1_{y} \, \supsetneq\, V^2_{y}\,=\,0
$$
and the trivial parabolic weight $\lambda_1\,=\, 0$.

Now for any real number $0\, \leq\, c\,\leq\, 1$, define
$$
(\varphi_* V)^c_x\, :=\, \bigoplus_{y\in \varphi^{-1}(x)} V^c_y
\, \subset\, \bigoplus_{y\in \varphi^{-1}(x)} V_y\,=\, (\varphi_* V)_x\, ,
$$
where the subspace $V^c_y\, \subset\, V_y$ is defined in \eqref{f1}. As mentioned following
\eqref{f1}, this filtration of subspaces $\{(\varphi_* V)^c_x\}_{0\leq c\leq 1}$ of
$(\varphi_* V)_x$ produces a parabolic structure on $\varphi_*V$ at the
point $x$ (see \cite[p.~80]{MY}).

Now take any $x\, \in \,\varphi(R)$, where $R$ is the divisor in \eqref{dr}. Let
$$\varphi^{-1}(x)_{\rm red} \,=\, \{y_1,\,\cdots,\, y_m\}\, \subset\, Y$$ be
the reduced inverse image of $x$ under $\varphi$, so $y_1,\,\cdots,\, y_m$ are distinct
points of $Y$. The multiplicity of $\varphi$ at $y_k$, $1\,\leq\, k\, \leq\, m$, will be
denoted by $b_k$; therefore, we have $$\varphi^{-1}(x)\,=\, \sum_{k=1}^m b_ky_k\, .$$

The projection formula says that
$$
\varphi_* (V\otimes {\mathcal O}_Y(-\sum_{j=1}^m b_jy_j))\,=\,
(\varphi_* V)\otimes {\mathcal O}_X(-x)\, .
$$
Hence the composition of homomorphisms
\begin{equation}\label{f2}
\varphi_* (V\otimes {\mathcal O}_Y(-\sum_{j=1}^m b_jy_j))\, \longrightarrow\,
\varphi_* V\, \longrightarrow\, (\varphi_* V)_x
\end{equation}
vanishes identically, where the first homomorphism is the natural inclusion map, and the
second homomorphism is the restriction of the vector bundle
$\varphi_* V$ to its fiber $(\varphi_* V)_x$ over $x$.
For every $1\, \leq\, k\, \leq\, m$, let
\begin{equation}\label{f3}
W_k\, \subset\,(\varphi_*V)_x
\end{equation}
be the image of the composition homomorphism
$$
\varphi_* (V\otimes {\mathcal O}_Y(-\sum_{j=1, j\not= k}^m b_jy_j))\, \longrightarrow\,
\varphi_* V\, \longrightarrow\, (\varphi_* V)_x\, ,
$$
where the first homomorphism is the natural inclusion map, and the
second homomorphism is the one in \eqref{f2}. It is straight-forward
to check that
$$
\dim W_k\,=\, b_k\cdot \dim V_{y_k}\, ,
$$
and
\begin{equation}\label{g1}
(\varphi_* V)_x\,=\, \bigoplus_{k=1}^m W_k\, .
\end{equation}
We will construct a weighted filtration on each of these vector spaces $W_k$; after that,
all these weighted filtrations combined together will give the weighted
filtration of $(\varphi_*V)_x$ using \eqref{g1}, which in turn
would define the parabolic structure on $\varphi_*V$ over the point $x$.

For each integer $0\,\leq\, \ell\,\leq\, b_k$, let
$$F^k_\ell\,\subset\, (\varphi_*V)_x$$ be the image of the composition homomorphism
\begin{equation}\label{g2}
\varphi_* (V\otimes {\mathcal O}_Y(-\ell\cdot y_k -\sum_{j=1, j\not= k}^m b_jy_j))\,
\longrightarrow\, \varphi_*V\, \longrightarrow\, (\varphi_* V)_x\, ,
\end{equation}
where the first homomorphism is the natural inclusion map, and the
second homomorphism is the one in \eqref{f2}.
Note that $F^k_{b_k} \,=\, 0$ (as the composition homomorphism
in \eqref{f2} vanishes), and $F^k_0 \,=\, W_k$ (constructed
in \eqref{f3}); in particular, we have $F^k_\ell\,
\subset\, W_k$ for all $0\,\leq\, \ell\,\leq\, b_k$. Therefore, we have a filtration of subspaces
\begin{equation}\label{f4}
W_k\,=\,F^k_0 \, \supset\, F^k_1\, \supset\, F^k_2\, \supset\, \cdots\, \supset\,
F^k_\ell \, \supset\, \cdots\, \supset\, F^k_{b_k-1}\, \supset\, F^k_{b_k}\,=\, 0\, .
\end{equation}

We will now show that every $0\, \leq\, \ell\, \leq\, b_k-1$,
the quotient space $F^k_\ell/F^k_{\ell+1}$ in \eqref{f4} is canonically
identified with the tensor product $V_{y_k}\otimes (K^{\otimes \ell}_Y)_{y_k}$, where $V_{y_k}$ and
$(K^{\otimes \ell}_Y)_{y_k}$ are the fibers, over the point $y_k\, \in\, \varphi^{-1}(x)$, of
$V$ and $K^{\otimes \ell}_Y$ respectively ($K_Y$ is the holomorphic cotangent bundle
of $Y$).

Take a holomorphic function $\beta$ defined on a sufficiently small open neighborhood $U$ of $y_k$ in $Y$ such that
$\beta(y_k) \,=\, 0$ and $d\beta (y_k)\, \not=\, 0$. We take $U$ such that
$U\bigcap R\,=\, y_k$, where $R$ is the ramification divisor in \eqref{dr}. Take any $v\,\in\, V_{y_k}$.
Let $s_v$ be a holomorphic section of $V$, defined on the neighborhood $U$ of $y_k$ in $Y$,
such that $$s_v(y_k) \,=\, v\, .$$ Now $\beta^\ell\cdot s_v$ defines a holomorphic section
of $\varphi_* (V\otimes {\mathcal O}_Y(-\ell\cdot y_k -\sum_{j=1, j\not= k}^m b_jy_j))$
over the open subset $\varphi(U)\, \subset\, X$. Let
\begin{equation}\label{1g6}
f'_{\beta,s_v}\,\in\,
\varphi_* (V\otimes {\mathcal O}_Y(-\ell\cdot y_k -\sum_{j=1, j\not= k}^m b_jy_j))_x
\end{equation}
be the evaluation, at the point $x\, \in\, X$, of this section of the
vector bundle $\varphi_* (V\otimes {\mathcal O}_Y(-\ell\cdot y_k -\sum_{j=1, j\not=
k}^m b_jy_j))$ over $X$. Since $F^k_\ell$ is the image of the
composition in \eqref{g2}, it follows that $F^k_\ell$ is a quotient of the fiber
$\varphi_* (V\otimes {\mathcal O}_Y(-\ell\cdot y_k -\sum_{j=1, j\not= k}^m b_jy_j))_x$.
So $f'_{\beta,s_v}$, constructed in \eqref{1g6}, produces an element of $F^k_\ell$. Let
\begin{equation}\label{g6}
f_{\beta,s_v}\, \in\, F^k_\ell/F^k_{\ell+1}
\end{equation}
be the image of $f'_{\beta,s_v}$ in the quotient space 
$F^k_\ell/F^k_{\ell+1}$ in \eqref{f4}.

We will first show that $f_{\beta,s_v}$ in \eqref{g6} is
independent of the choice of the section $s_v$ passing through $v$.

If $t_v$ is another holomorphic section of $V$, defined on $U$,
such that $t_v(y_k) \,=\, v$. Then $(s_v-t_v)(y_k)\,=\, 0$, and therefore,
the section $\beta^\ell\cdot (s_v-t_v)$ vanishes at $y_k$ of order
at least $\ell+1$. This implies that the section of
$\varphi_* (V\otimes {\mathcal O}_Y(-\ell\cdot y_k -\sum_{j=1, j\not= k}^m b_jy_j))$
over $\varphi(U)$ defined by $\beta^\ell\cdot (s_v-t_v)$ lies
in the image of the inclusion map
\begin{equation}\label{g5}
\varphi_* (V\otimes {\mathcal O}_Y(-(\ell+1)\cdot y_k -\sum_{j=1, j\not= k}^m b_jy_j))_x
\, \hookrightarrow\,\varphi_* (V\otimes {\mathcal O}_Y(-\ell\cdot y_k -\sum_{j=1, j\not= k}^m b_jy_j))\, .
\end{equation}

Consider the element of the fiber $\varphi_* (V\otimes {\mathcal O}_Y(-\ell\cdot y_k -
\sum_{j=1, j\not= k}^m b_jy_j))_x$ obtained by 
evaluating, at the point $x$, of the above section of the vector
bundle $$\varphi_* (V\otimes {\mathcal O}_Y(-\ell\cdot y_k -\sum_{j=1, j\not= k}^m b_jy_j))$$
defined by $\beta^\ell\cdot (s_v-t_v)$. Since $\beta^\ell\cdot (s_v-t_v)$ lies in the
sub-sheaf in \eqref{g5}, it follows that the image of this element of
$\varphi_* (V\otimes {\mathcal O}_Y(-\ell\cdot y_k -
\sum_{j=1, j\not= k}^m b_jy_j))_x$ in $F^k_\ell$ actually lies in the
subspace $F^k_{\ell+1}\, \subset\, F^k_\ell$.
Therefore, the element $f_{\beta,s_v}$ in \eqref{g6} coincides with $f_{\beta,t_v}$
constructed similarly. Hence it follows that $f_{\beta,s_v}$ does not depend on the 
choice of the section $s_v$ passing through $v$.

Regarding the dependence of $f_{\beta,s_v}$ on $\beta$, 
we will now show that $f_{\beta,s_v}$ depends only on $(d\beta (y_k))^{\otimes\ell}
\,\in\, (K^{\otimes\ell}_Y)_{y_k}$.

To prove this, take another holomorphic function $\beta_1$ defined around
$y_k$ such that $\beta_1(y_k) \,=\, 0$ and $(d\beta_1 (y_k))^{\otimes \ell}\,=\,
(d\beta(y_k))^{\otimes \ell}$.
Then the function $\beta^\ell - \beta^\ell_1$ vanishes at $y_k$ of order at least
$\ell+1$. Consequently, the local section $(\beta^\ell-\beta^\ell_1)\cdot s_v$
of $V$ vanishes at $y_k$ of order at least $\ell+1$. Form this it is straight-forward
to deduce that 
$$
f_{\beta,s_v}\,=\, f_{\beta_1,s_v}\, \in\, F^k_\ell/F^k_{\ell+1}\, ,
$$
where $f_{\beta_1,s_v}$ is constructed as done in \eqref{g6}.

Consequently, for every $0\,\leq\, \ell\,<\, b_k$, we get a homomorphism
\begin{equation}\label{nu}
\nu(y_k, \ell)\, :\, V_{y_k}\otimes (K^{\otimes \ell}_Y)_{y_k}\, \longrightarrow\,
F^k_\ell/F^k_{\ell+1}\, ,\ \ v\otimes w\, \longmapsto\, f_{\beta,s_v}\, ,
\end{equation}
where $w\,=\, d\beta (y_k)^{\otimes \ell}$ and $s_v(y_k)\,=\, v$. The homomorphism
$\nu(y_k, \ell)$ in \eqref{nu} is evidently an isomorphism.

Now, this $y_k$ may be a parabolic point of $V_*$, in which case $V_{y_k}$ would have a
quasi-parabolic filtration. If $y_k$ is a parabolic point of $V_*$, let
\begin{equation}\label{f5}
V_{y_k}\,=\, V^1_{y_k} \, \supsetneq\, V^2_{y_k} \, \supsetneq\, \cdots \, \supsetneq\,
V^{l(y_k)}_{y_k} \, \supsetneq\,V^{l(y_k)+1}_{y_k}\,=\, 0
\end{equation}
be the quasi-parabolic filtration of the fiber $V_{y_k}$, and let 
\begin{equation}\label{e4}
0\, \leq \, \lambda^1_{y_k}\, < \, \cdots \, < \, \lambda^{l(y_k)}_{y_k} \, < \,
\lambda^{l(y_k)+1}_{y_k} \, =\, 1
\end{equation}
be the corresponding system of parabolic weights.

If $y_k$ is not a parabolic point of $V_*$, then we
equip $V_{y_k}$ with the trivial quasi-parabolic filtration
$$
V_{y_k}\,=\, V^1_{y_k} \, \supsetneq\, V^2_{y_k}\,=\,0 
$$
and the trivial parabolic weight $\lambda^1_{y_k}\,=\, 0$. This way we would not need
to distinguish between parabolic $y_k$ and non-parabolic $y_k$.

Using the isomorphism $\nu(y_k, \ell)$ in \eqref{nu}, the filtration in \eqref{f5} 
produces a filtration of subspaces $F^k_\ell/F^k_{\ell+1}$ for every $0\, \leq\, \ell\, 
\leq\, b_k-1$, because $(K^{\otimes \ell}_Y)_{y_k}$ is a line, so linear subspaces of 
$V_{y_k}$ are in a canonical bijection with the linear subspaces of $V_{y_k}\otimes 
(K^{\otimes \ell}_Y)_{y_k}$. This filtration of each $F^k_\ell/F^k_{\ell+1}$, $0\, \leq\, 
\ell\, \leq\, b_k-1$, and the filtration of $W_k$ in \eqref{f4} together produce a finer 
filtration $W_k$, while the parabolic weights in \eqref{e4} produce weights for the terms 
of this finer filtration of $W_k$.

More precisely, the length of this finer filtration of $W_k$
$$
W_k\,=\, S_1\, \supsetneq\, S_2\, \supsetneq\, \cdots\, \supsetneq\,
S_{l(y_k)b_k-1}\, \supsetneq\, S_{l(y_k)b_k} \, \supsetneq\, S_{l(y_k)b_k+1}\, =\, 0
$$
is $l(y_k)b_k$, and for every $1\, \leq\, i\,
\,\leq\, l(y_k)b_k$, the $i$-th term $S_i$, and its weight, are constructed as
follows: First write $i\,=\, c\cdot l(y_k) + d$, where $c$ is a nonnegative integer and
$1\,\leq\, d\,\leq\, l(y_k)$.
\begin{itemize}
\item If $d\,=\, 1$, then set $$S_i\,:=\, F^k_c$$
(see \eqref{f4}), and set the weight of $S_i$ to be $(c+\lambda^1_{y_k})/b_k$,
where $\lambda^1_{y_k}$ is the weight in \eqref{e4}.

\item If $2\, \leq\, d\, \leq\, l(y_k)$, then $S_i$ satisfies the two
conditions
\begin{enumerate}
\item{} $F^k_{c+1}\, \subsetneq\, S_i\, \subsetneq\, F^k_c$, and

\item{} the subspace $S_i/F^k_{c+1}\,\subset\, F^k_c/F^k_{c+1}\,=\, V_{y_k}\otimes
K^{\otimes c}_Y$ (it is the isomorphism
$\nu(y_k, c)$ in \eqref{nu}) coincides with the subspace $V^d_{y_k}\otimes
K^{\otimes c}_Y \, \subset\, V_{y_k}\otimes K^{\otimes c}_Y$
(see \eqref{f5} for $V^d_{y_k}$).
\end{enumerate}
The above two conditions evidently determine $S_i$ uniquely. The weight of $S_i$ is
set to be
\begin{equation}\label{dk}
(c+\lambda^d_{y_k})/b_k\, ,
\end{equation}
where $\lambda^d_{y_k}$ is the weight in \eqref{e4}.
\end{itemize}

Now that we have the weighted filtration of $W_k$ constructed above, using the 
decomposition in \eqref{g1} we construct a parabolic structure on $(\varphi_* V)_x$ in 
the obvious way: The subspace of $(\varphi_* V)_x$ corresponding to any $0\, \leq\, c\, 
\leq\, 1$ is simply the direct sum of the subspaces of $W_k$, $1\, \leq\, k\, \leq\, m$, 
corresponding to the same $c$.

Therefore, we get a parabolic structure on the vector bundle $\varphi_* V$ over $X$.
This parabolic bundle over $X$ will be denoted by $\varphi_*V_*$.

\section{Direct image of parabolic bundles and ramified torus sub-bundles}\label{se3}

\subsection{Ramified torus sub-bundle}\label{se2.2}

Let $V_0$ be a finite dimensional complex vector space. Consider the group $\text{GL}(V_0)$ 
consisting of all linear automorphisms of $V_0$; it is a complex affine algebraic group. A torus 
in $\text{GL}(V_0)$ is a complex algebraic subgroup of $\text{GL}(V_0)$ isomorphic 
to a product of copies of the multiplicative group ${\mathbb C}^{\times}\,=\, {\mathbb 
C}\setminus\{0\}$. Given a torus $T_0\, \subset\, \text{GL}(V_0)$ consider the isotypical
decomposition of $V_0$ for the standard action of $T_0$ on $V_0$; it consists of finitely many
distinct characters $\chi_1,\, \cdots,\, \chi_n$ of $T_0$ and a decomposition 
$$
V_0\,=\, \bigoplus_{j=1}^n V_{0,j}
$$
such that any $t\, \in\, T_0$ acts on $V_{0,j}$ as multiplication by $\chi_j(t)\,\in\,
{\mathbb C}^{\times}$. The centralizer of $T_0$ in $\text{GL}(V_0)$ will be denoted by
$C(T_0)$. It is straight-forward to see that $C(T_0)$ consists of all $\alpha\, \in\,
\text{GL}(V_0)$ such that $\alpha(V_{0,j})\,=\, V_{0,j}$ for every $1\, \leq\,j\, \leq\, n$.

Conversely, given a direct sum decomposition
\begin{equation}\label{c1}
V_0\,=\, \bigoplus_{j=1}^m B_j\, ,
\end{equation}
let ${\mathbb L}\, \subset\,\text{GL}(V_0)$ be the subgroup consisting of all $\alpha\, \in\,
\text{GL}(V_0)$ such that $\alpha(B_j)\,=\, B_j$ for every $1\, \leq\,j\, \leq\, m$.
Then the center of $\mathbb L$ is a torus in $\text{GL}(V_0)$. The decomposition of
$V_0$ corresponding to this torus coincides with the one in \eqref{c1}.

Let $Z$ be a smooth complex algebraic curve; it need not be projective or connected.
Take an algebraic vector bundle $W$ on $Z$. Consider
the vector bundle $\text{ad}(W)\,=\, \text{End}(W)\,=\, W\otimes W^*$.
We have a function on the total space of $\text{ad}(W)$ given by the determinant of
endomorphisms. Let
$$
\text{Ad}(W)\, \subset\, \text{ad}(W)
$$
be the Zariski open subset where the determinant is nonzero, so $\text{Ad}(W)$
parametrizes the automorphisms. Let
$$
p_0\, :\, \text{Ad}(W)\,\longrightarrow\, Z
$$
be the projection. So the fiber $\text{Ad}(W)_y\,=\, p^{-1}_0(y)$ of $\text{Ad}(W)$
over any $y\, \in\, Z$ is the space of all linear automorphisms of the fiber $W_y$.
Note that $\text{Ad}(W)$ is smooth group-scheme over $Z$.

\begin{definition}
A \textit{torus sub-bundle} of $\text{Ad}(W)$ is a smooth abelian semi-simple subgroup scheme 
over $Z$ $${\mathcal T}\,\subset\, \text{Ad}(W)\,.$$
In other words, the restriction of 
$p_0$ to $\mathcal T$ is smooth, and for every $y\, \in\, Z$, the fiber ${\mathcal 
T}_y$ is a torus in $\text{Ad}(W)_y$.
\end{definition}

Note that the dimension of ${\mathcal T}_y$ 
may depend on the component of $Z$ in which $y$ lies. In other words, the connected 
components of $\mathcal T$ are allowed to have different dimensions.

As mentioned earlier, we have the isotypical decomposition of $W_y$ for the
standard action of ${\mathcal T}_y$ on $W_y$. It should be noted that these point-wise
decomposition of the fibers of $W$ does not produce a decomposition of
the vector bundle $W$ because
the direct summands may get interchanged as $y$ runs over $Z$. To gave such an example,
take $X$ to be an irreducible smooth complex projective curve of sufficiently large
genus. Let $P_m$ be the group of permutations of $\{1,\, 2,\, 3,\, \cdots,\,m\}$, with
$m\, \geq\, 3$; it acts on ${\mathbb C}^m$ by permuting the elements
of the standard basis of ${\mathbb C}^m$.
Let $$\rho\, :\, \pi_1(X,\, z_0)\,\longrightarrow\, P_m$$
be a surjective homomorphism, and let $W$ be the flat vector bundle on $X$
of rank $m$ associated to the composition homomorphism
$$
\pi_1(X,\, z_0)\,\stackrel{\rho}{\longrightarrow}\, P_m\, \hookrightarrow\,
\text{GL}(m, {\mathbb C})\, .
$$
Each fiber of $W$ has a canonical decomposition into an unordered direct sum of lines.
However, the vector bundle $W$ is not a direct sum of line bundles. In fact, it is the direct sum of the
trivial line bundle on $X$ and an indecomposable bundle of rank $m-1$.

Now let $W_*$ be a parabolic vector bundle $Z$. The underlying vector bundle
is denoted by $W$, and the parabolic divisor is $\{z_1,\, \cdots, \, z_n\}$.
For each $1\, \leq\, i\, \leq\, n$, the quasi-parabolic filtration on $W_{z_i}$ is
$$
W_{z_i}\,=\, F^i_1 \, \supsetneq\, F^i_2\, \supsetneq\, \cdots \, \supsetneq\,
F^i_{\ell_i}\, \supsetneq\, F^i_{\ell_i+1}\,=\, 0\, ,
$$
while the corresponding parabolic weights are
$$
0\, \leq \, \lambda^i_1\, <\, \lambda^i_2\, <\, \cdots\, < \,
\lambda^i_{\ell_i-1}\, < \, \lambda^i_{\ell_i} \, <\, \lambda^i_{\ell_i+1}\,=\, 1\, .
$$

We assume that all the parabolic weights $\lambda^i_j$ are rational numbers.

There is an equivalence of categories between the parabolic vector bundles with rational
parabolic weights and the equivariant bundles on some ramified Galois covering \cite{Bi1},
\cite{Bo1}. We very briefly recall the construction of this equivalence of categories.

Let $\gamma^0\, :\, Z_1\, \longrightarrow\, Z$ be a (possibly) ramified Galois covering with Galois
group $\Gamma_0\,=\, \text{Gal}(\gamma^0)$. Let $D_0\, \subset\, Z_1$ be the reduced effective divisor 
where the map $\gamma^0$ is ramified. Let $W$ be a $\Gamma_0$--equivariant vector bundle on
$Z_1$. The group $\Gamma_0$ acts on the direct image $\gamma^0_* W$. Consider the vector bundle
$$
(\gamma^0_* W)^{\Gamma_0}\, \subset\, \gamma^0_* W \, \longrightarrow\, Z
$$
given by the $\Gamma_0$--invariant part of the direct image.
This vector bundle $(\gamma^0_* W)^{\Gamma_0}$ has a filtration of sub-sheaves given by the invariant direct images
$$
(\gamma^0_* (W\otimes {\mathcal O}_{Z_1}(-i\cdot D)))^{\Gamma_0}\, \subset\, \gamma^0_* (W\otimes {\mathcal O}_{Z_1}(-i\cdot D))
\, \longrightarrow\, Z\, , \ \ i \, \geq\, 1\, .
$$
This filtration of sub-sheaves of $(\gamma^0_* W)^{\Gamma_0}$ produces a parabolic structure on the vector bundle $(\gamma^0_* W)^{\Gamma_0}$. The
parabolic weights are rational numbers. Conversely, given a a parabolic bundle $E_*$ on $Z$ with rational parabolic weights, there is
a (possibly) ramified Galois covering $\gamma^0\, :\, Z_1\, \longrightarrow\, Z$, and a
$\text{Gal}(\gamma^0)$--equivariant vector bundle $W$ on $Z_1$, such that the parabolic vector bundle on $Z$ corresponding 
to $W$ is isomorphic to $E_*$.

It should be mentioned that the above equivalence is local in
$Z$, both in analytic and Zariski topology. In other words, $Z$ can be replaced by
an open subset of it in this equivalence of categories. We also note that outside
the parabolic points, the equivariant bundle corresponding to a parabolic bundle is
simply the pullback of the vector bundle underlying the parabolic bundle.

The complement $Z\setminus \{z_1,\, \cdots, \, z_n\}$ will be denoted by $U$.

Fix a point $y\, \in\, \{z_1,\, \cdots, \, z_n\}$, and take an
open neighborhood $y\, \in\, U_y\, \subset\, Z$ of $y$ in $Z$, in analytic
or Zariski topology. Consider an equivariant bundle $(\widetilde{U}_y,\,\gamma,\,
W^y)$ corresponding to the parabolic bundle $W_*\vert_{U_y}$. This means that
$$
\gamma\, :\, \widetilde{U}_y\, \longrightarrow\, U_y
$$
is a ramified Galois covering morphism, and $W^y$ is a $\text{Gal}(\gamma)$--equivariant
vector bundle on $\widetilde{U}_y$ that corresponds to $W_*\vert_{U_y}$. As mentioned
before, the restriction of $W^y$ to $\gamma^{-1}(U_y\bigcap U)$ is identified with
the restriction of $\gamma^*W$ to $\gamma^{-1}(U_y\bigcap U)$. We also note that
$\gamma^* \text{Ad}(W)$ is canonically identified with $\text{Ad}(\gamma^* W)$.

\begin{definition}\label{drts}
A \textit{ramified torus sub-bundle} for $W_*$ is a torus sub-bundle
$${\mathcal T}\, \subset\, \text{Ad}(W)\vert_U$$ over $U$ satisfying the following
condition: For every point $y\, \in\, \{z_1,\, \cdots, \, z_n\}$, there is an
open neighborhood $y\, \in\, U_y\, \subset\, Z$ of $y$ in $Z$, such that if
$(\widetilde{U}_y,\, \gamma,\, W^y)$ is an equivariant bundle corresponding to the
parabolic bundle $W_*\vert_{U_y}$, then the torus sub-bundle
$$
\gamma^*{\mathcal T}\, \subset\, \text{Ad}(\gamma^* W)\,=\, \text{Ad}(W^y)
$$
on $\gamma^{-1}(U_y\bigcap U)$ extends across $\gamma^{-1}(y)$ to produce a torus sub-bundle
of $\gamma^* \text{Ad}(W^y)$ over the open subset $\gamma^{-1}(U_y\bigcap
(U\bigcup \{\gamma^{-1}(y)\}))$.
\end{definition}

We have used above that $W^y$ and $\gamma^*W$ are identified over 
$\gamma^{-1}(U_y\bigcap U)$.

In the above definition it is equivalent to take the open neighborhood $U_y$ to be in
Zariski topology or in analytic topology. More precisely, if there is an analytic
open subset $U_y$ satisfying the above conditions, then there is a also Zariski open
subset $U_y$ satisfying the above conditions. Indeed, the only issue is the
order of ramification of $\gamma$ over $y$. Any given order can be achieved by
coverings over both analytic and Zariski neighborhoods of $y$.
In fact, $U_y$ can be taken to be a
ramified covering of $Z$ which works simultaneously for all points of $Z\setminus U$.
To explain this, we recall a result on coverings of surfaces.

Let $S$ be a connected compact oriented $C^\infty$ surface and
$$
\{x_1,\, \cdots,\, x_d\}\, \subset\, S
$$
a finite subset. For each $1\, \leq\, i\, \,\leq\, d$, fix an integer $n_i\, \geq\, 2$
for the point $x_i$. Assume that least
one of the following three conditions are satisfied:
\begin{enumerate}
\item{} ${\rm genus}(S)\, \geq\, 1$,

\item $d \, \notin \, \{1\, ,2\}$,

\item if $d \,=\, 2$, then $n_1\,=\,n_2$.
\end{enumerate}
So only two cases are ruled out by this assumption: $\text{genus}(S)\,=\,0\,=\, d-1$
and $\text{genus}(S)\,=\,0\, =\, d-2\,\not=\, n_1-n_2$. A theorem due to
Bundgaard--Nielsen and Fox says that there is a finite Galois covering
$$
\beta\, :\, \widetilde{S}\, \longrightarrow\, S
$$
such that $\beta$ is unramified over $S\setminus \{x_1,\, \cdots,\, x_d\}$,
and for every $1\, \leq\, i\, \,\leq\, d$, the order of ramification
at every point of $\beta^{-1}(x_i)$ is $n_i$
\cite[p. 26, Proposition 1.2.12]{Na}. (We call
the order of ramification at $0$ of the map $z\, \longmapsto\,
z^k$ to be $k$.)

In view of the above theorem, given a ramified torus sub-bundle ${\mathcal T}\,
\subset\, \text{Ad}(W)\vert_U$ for $W_*$,
there is a ramified Galois covering 
$$
\gamma\, :\, \widetilde{Z}\, \longrightarrow\, Z
$$
\'etale over $U$, and a $\text{Gal}(\gamma)$--equivariant vector bundle $\widetilde{W}$
on $\widetilde{Z}$ that corresponds to $W_*$, such that the torus sub-bundle
$$\gamma^*{\mathcal T}\, \subset\, \text{Ad}(\gamma^*W)$$ over $\gamma^{-1}(U)$
extends to entire $\widetilde{Z}$ as a torus sub-bundle of $\text{Ad}(\widetilde{W})$.

If we are in a situation of the two exceptions in the above theorem of Bundgaard--Nielsen 
and Fox, simply introduce either one or two extra parabolic points with trivial parabolic 
structure over those points (meaning trivial quasi-parabolic filtration and zero parabolic 
weight).

It should be mentioned that given a ramified torus sub-bundle
$${\mathcal T}\, \subset\, \text{Ad}(W)\vert_U$$
for $W_*$, this ${\mathcal T}$ does not, in general, extend to some torus of $\text{Ad}(W)_{z_i}$
on a parabolic point $z_i$. However $\mathcal T$ produces a parabolic subgroup
of $\text{Ad}(W)_{z_i}$.

\subsection{Ramified torus sub-bundle for direct image of a vector bundle}\label{se3.2}

As in Section \ref{se2.2}, let $Z$ be a smooth complex algebraic curve,
which may not be projective or connected.
Let $Y$ be a smooth complex algebraic curve and $\phi\, :\, Y\, \longrightarrow\, Z$
an \'etale covering; so $Y$ also need not be connected or projective. Take a
vector bundle $V$ on $Y$, and consider the vector bundle $$W\,=\, \phi_*V$$ on $Z$.
For any $x\, \in\, Z$, the fiber $W_x$ has the direct sum decomposition
\begin{equation}\label{t2}
W_x\,=\, (\phi_*V)_x\,=\, \bigoplus_{y\in \phi^{-1}(x)} V_y\, .
\end{equation}
This produces a torus sub-bundle $\mathcal T$ of $\text{Ad}(W)$.

To describe $\mathcal T$ in another way, note that the direct image $\phi_*\text{Ad}(V)$
is a subgroup scheme
$$
\phi_*\text{Ad}(V)\, \subset\, \text{Ad}(W)\, .
$$
In fact, we have $\phi_*\text{Ad}(V)\, =\, (\phi_*\text{End}(V))\bigcap \text{Ad}(W)$.
The center of $\phi_*\text{Ad}(V)$ is the torus sub-bundle
\begin{equation}\label{dT}
{\mathcal T}\, \subset\, \phi_*\text{Ad}(V)\, \subset\, \text{Ad}(W)\, .
\end{equation}
Note that
the centralizer of $\mathcal T$ in $\text{Ad}(W)$ is $\phi_*\text{Ad}(V)$. Related
constructions can be found in \cite{DG}, \cite{DoPa}.

Now assume that $Z$ is connected and projective, and allow the covering $\phi$ to be ramified,
but assume that $Y$ is also connected; note that $Y$ is projective because $Z$ is so.
Let $U\, \subset\, Z$ be the Zariski open dense subset over which $\phi$
is \'etale.

As before, $W\,=\, \phi_*V$ is a vector bundle on $Z$, because $\phi$ is a finite
morphism. As constructed in \eqref{dT}, we have a torus sub-bundle
\begin{equation}\label{t1}
{\mathcal T}\, \subset\, \text{Ad}(\phi_*V)\vert_U\,=\, \text{Ad}(W)\vert_U
\end{equation}
over $U$.

In Section \ref{se2} we saw that $W$ has a parabolic structure on the complement of 
$U$; we are taking the trivial parabolic structure on $V$. Let $W_*$ be the parabolic 
vector bundle defined by this parabolic structure with $W$ as the underlying vector 
bundle. Note that all the parabolic weights of $W_*$ are rational numbers.

\begin{proposition}\label{prop1}
The torus sub-bundle ${\mathcal T}$ in \eqref{t1} is a ramified torus sub-bundle
for $W_*$.
\end{proposition}

\begin{proof}
There is a ramified covering $\widetilde{\phi}\, :\, \widetilde{Y}\, \longrightarrow\, 
Y$ such that the composition $$\phi\circ\widetilde{\phi}\, :\, \widetilde{Y}\, 
\longrightarrow\, Z$$ is Galois. Indeed, this follows immediately from the fact that 
any finite index subgroup of a finitely presented group $G$ contains a normal subgroup 
of $G$ of finite index. Note that this implies that the covering $\widetilde{\phi}$ is 
Galois. Let $G$ (respectively, $H$) denote the Galois group of 
$\phi\circ\widetilde{\phi}$ (respectively, $\widetilde{\phi}$). The quotient map $G\, 
\longrightarrow\, G/H$ will be denoted by $q$. Fix a finite subset $C$ of $G$ such that 
the composition of maps
$$
C\, \hookrightarrow\, G \, \stackrel{q}{\longrightarrow}\, G/H
$$
is a bijection. The left--translation action of $G$ on $G/H$ produces an action of $G$
on $C$ using the above bijection.

Consider the vector bundle
\begin{equation}\label{tv}
\widetilde{V}\, :=\, 
\bigoplus_{\sigma\in C} (\sigma^{-1})^*\widetilde{\phi}^* V\, \longrightarrow\, \widetilde{Y}\, .
\end{equation}
Using the above action of $G$ on $C$, the vector bundle $\widetilde{V}$ in \eqref{tv} has the 
structure of a $G$--equivariant vector bundle on $\widetilde Y$. The parabolic vector bundle on 
$\widetilde{Y}/G\,=\, Z$ corresponding to this $G$--equivariant vector bundle $\widetilde V$ is 
in fact $W_*$. (This is straight-forward to deduce from the correspondence between parabolic 
bundles and equivariant bundles.) In particular, we have
\begin{equation}\label{t3}
((\phi\circ\widetilde{\phi})^* W)\vert_{(\phi\circ\widetilde{\phi})^{-1}(U)}\,=\, 
\widetilde{V}\vert_{(\phi\circ\widetilde{\phi})^{-1}(U)}\, .
\end{equation}

Over $U\, \subset\, Z$, consider the decomposition of the fibers of $W\vert_U\,=\,
(\phi_*V)\vert_U$ given be the torus sub-bundle $\mathcal T$ in \eqref{t1}. We recall that
for any $x\, \in\, U$, this decomposition is the natural direct sum decomposition
of $W_x$ in \eqref{t2}. Using the isomorphism in
\eqref{t3}, the pullback of this decomposition of the fibers of
$W\vert_U$ produces a decomposition of the fibers of $\widetilde{V}$ over
$(\phi\circ\widetilde{\phi})^{-1}(U)$. It is straight-forward to check that this
decomposition of the fibers of $\widetilde{V}\vert_{(\phi\circ\widetilde{\phi})^{-1}(U)}$
coincides with the decomposition of $\widetilde{V}\vert_{(\phi\circ\widetilde{\phi})^{-1}(U)}$
in \eqref{tv}. But the decomposition of $\widetilde{V}$ in \eqref{tv} is over entire
$\widetilde Y$. This immediately implies that the torus sub-bundle
$$
(\phi\circ\widetilde{\phi})^*{\mathcal T} \,\subset\,
(\phi\circ\widetilde{\phi})^*(\text{Ad}(W)\vert_U)\,=\,
\text{Ad}(\widetilde{V})\vert_{(\phi\circ\widetilde{\phi})^{-1}(U)}
$$
(see \eqref{t3} and \eqref{t1}) extends to a torus sub-bundle of
$\text{Ad}(\widetilde{V})$ over entire $\widetilde Y$.
Hence ${\mathcal T}$ is a ramified torus sub-bundle for $W_*$ (see Definition
\ref{drts}).
\end{proof}

\subsection{Ramified torus sub-bundle for direct image of a parabolic bundle}\label{se3.3}

As before, $Y$ and $Z$ are smooth connected projective complex curves, and $\phi$
is a finite morphism from $Y$ to $Z$. Let $V_*$ be a parabolic vector bundle over $Y$
with parabolic divisor $\{y_1,\, \cdots, \, y_n\}\, \subset\, Y$;
for each $1\, \leq\, i\, \leq\, n$, the quasi-parabolic filtration on $V_{y_i}$ is
$$
V_{y_i}\,=\, F^i_1 \, \supsetneq\, F^i_2\, \supsetneq\, \cdots \, \supsetneq\,
F^i_{\ell_i}\, \supsetneq\, F^i_{\ell_i+1}\,=\, 0\, ,
$$
while the corresponding parabolic weights are
$$
0\, \leq \, \lambda^i_1\, <\, \lambda^i_2\, <\, \cdots\, < \,
\lambda^i_{\ell_i-1}\, < \, \lambda^i_{\ell_i} \, <\, \lambda^i_{\ell_i+1}\,=\, 1\, .
$$
We assume that all the parabolic weights $\lambda^i_j$ are rational numbers.

Let $$W_*\,:=\, \phi_* V_*$$ be the parabolic vector bundle over $Z$ (constructed as done
in Section \ref{se2}). Let
\begin{equation}\label{du}
U\, \subset\, Z\setminus \phi(\{y_1,\, \cdots, \, y_n\})\, \subset\, Z
\end{equation}
be the Zariski open dense subset over which the restriction of $\phi$
to $\phi^{-1}(Z\setminus \phi(\{y_1, \cdots, y_n\}))$ is \'etale. We have a torus
sub-bundle
\begin{equation}\label{t11}
{\mathcal T}\, \subset\, \text{Ad}(\phi_*V)\vert_U\,=\, \text{Ad}(W)\vert_U\, ,
\end{equation}
where $V$ and $W$ are the vector bundle underlying $V_*$ and $W_*$ respectively.

The following is a generalization of Proposition \ref{prop1}.

\begin{proposition}\label{prop2}
The torus sub-bundle ${\mathcal T}$ in \eqref{t11} is a ramified torus sub-bundle
for the parabolic bundle $W_*\,:=\, \phi_* V_*$.
\end{proposition}

\begin{proof}
Let $\psi\, : \, \widehat{Z} \, \longrightarrow\, Z$ be a ramified Galois 
covering and $\widehat{W}$ a $\text{Gal}(\psi)$--equivariant vector bundle 
on $\widehat Z$ that corresponds to the parabolic vector bundle $W_*$; such a pair $(\psi,\, 
\widehat{W})$ exists because all the parabolic weights of $W_*$ are rational numbers,
which is in fact ensured by the assumption that all the parabolic weights of $V_*$ are rational
numbers. The Galois group $\text{Gal}(\psi)$ for $\psi$ will be denoted by $G$.

Let $\widetilde{Y}$ be the normalization of the fiber product $\widehat{Z}\times_Z Y$
over $Z$. Let
$$
{\widetilde\phi}\, :\, \widetilde{Y}\, \longrightarrow\, \widehat{Z}\ \
\text{ and }\ \ {\widetilde\psi}\, :\, \widetilde{Y}\, \longrightarrow\, Y
$$
be the natural projections. The above morphism ${\widetilde\psi}$ is a ramified
Galois covering with Galois group $G$. On the other hand, the morphism
${\widetilde\phi}$ is \'etale.

There is a $G$--equivariant vector bundle $\widetilde V$ on $\widetilde Y$ that
corresponds to the parabolic vector bundle $V_*$ on $Y\,=\, \widetilde{Y}/G$
(the Galois group for ${\widetilde\psi}$ is $G$). The action of $G$ on $\widetilde V$
produces an action of $G$ on the direct image ${\widetilde\phi}_*\widetilde{V}\,
\longrightarrow\, \widehat{Z}$. This action of $G$ on ${\widetilde\phi}_*\widetilde{V}$
commutes with the action of $G\,=\, \text{Gal}(\psi)$ on $\widehat Z$.

Consider the open subset $U\, \subset\, Z$ in \eqref{du}. It is straight-forward to
check that the restriction of ${\widetilde\phi}_*\widetilde{V}$ to $\psi^{-1}(U)$
is identified with the pullback $\psi^*(W_*\vert_U)\, \longrightarrow\, \psi^{-1}(U)$;
also this identification is $G$--equivariant for the natural action of $G\,=\,
\text{Gal}(\psi)$ on $\psi^*(W_*\vert_U)$.
On the other hand, the vector bundle $\widehat{W}\vert_{\psi^{-1}(U)}$ is
identified with $\psi^*(W_*\vert_U)\, \longrightarrow\, \psi^{-1}(U)$, because the
parabolic vector bundle $W_*$ corresponds to the $G$--equivariant vector bundle
$\widehat W$ on $\widehat Z$; this identification is also $G$--equivariant.

The resulting isomorphism between the two $G$--equivariant vector bundles 
$\widehat{W}\vert_{\psi^{-1}(U)}$ and 
$({\widetilde\phi}_*\widetilde{V})\vert_{\psi^{-1}(U)}$ over $\psi^{-1}(U)$ extends to 
an isomorphism of $G$--equivariant vector bundles between $\widehat W$ and 
${\widetilde\phi}_*\widetilde{V}$ over $\widehat Z$. In fact, both the $G$--equivariant 
vector bundles $\widehat W$ and ${\widetilde\phi}_*\widetilde{V}$ correspond to the 
parabolic vector bundle $W_*$ on $\widehat{Z}/G \,=\, Z$.

Since the morphism ${\widetilde\phi}$ is \'etale, for the direct image
${\widetilde\phi}_*\widetilde{V}$ there is a torus sub-bundle
$$
\widehat{\mathcal T}\, \subset\, \text{Ad}({\widetilde\phi}_*\widetilde{V})
$$
over $\widehat{Z}$.

Consider ${\mathcal T}$ over $U$ in \eqref{t11}. The above isomorphism between 
$\widehat W$ and ${\widetilde\phi}_*\widetilde{V}$ takes $\psi^*{\mathcal T}$ over 
$\psi^{-1}(U)$ to $\widehat{\mathcal T}\vert_{\psi^{-1}(U)}$. Consequently, 
$\psi^*{\mathcal T}$ extends to a torus sub-bundle of $\text{Ad}(\widehat{W})$ over 
entire $\widehat Z$.
\end{proof}

\section{Direct image, connection and ramified torus sub-bundle}\label{se4}

\subsection{Direct image of connection on parabolic bundle}

Let $\varphi\, :\, Y\, \longrightarrow\, X$ is a finite morphism between smooth complex 
projective curves, with $X$ connected. Let $V_*$ be a parabolic vector bundle on $Y$. Consider 
the parabolic vector bundle $\varphi_*V_*$ on $X$ constructed in Section \ref{se2}. We will show 
that a connection on the parabolic vector bundle $V_*$ induces a connection on $\varphi_*V_*$. 
We begin with a general observation.

\begin{lemma}\label{lem1}
Let $W$ be a vector bundle over $Y$ equipped with a logarithmic connection $D$ with
singular divisor $S_0$. Then $D$ induces a logarithmic connection on the vector
bundle $\varphi_*W\, \longrightarrow\, X$ whose
singular divisor is $\varphi(S_0\bigcup R)$, where $R\, \subset\, Y$ as in
\eqref{dr} is the ramification divisor for $\varphi$.
\end{lemma}

\begin{proof}
For notational convenience, the reduced divisor $\varphi(S_0\bigcup R)_{\rm red}$ will be
denoted by $S_X$. The reduced divisor $\varphi^{-1}(S_X)_{\rm red}$ will be denoted by
$S_Y$. It is straight-forward to check that
$$
\varphi^*(K_X\otimes {\mathcal O}_X(S_X))\,=\, K_Y\otimes {\mathcal O}_Y(S_Y)\, .
$$
Therefore, by the projection formula, we have
\begin{equation}\label{f6}
\varphi_*(W\otimes K_Y\otimes {\mathcal O}_Y(S_Y))\,=\,
(\varphi_*W)\otimes K_X\otimes {\mathcal O}_X(S_X)\, .
\end{equation}

Consider the composition homomorphism
$$
W\, \stackrel{D}{\longrightarrow}\, W\otimes K_Y\otimes {\mathcal O}_Y(S_0)\,
\hookrightarrow\, W\otimes K_Y\otimes {\mathcal O}_Y(S_Y)\, ;
$$
note that $S_0\, \subset \, S_Y$. Now taking direct image of this composition, and
using \eqref{f6}, we have
$$
\varphi_*W\, \longrightarrow\, \varphi_*(W\otimes K_Y\otimes {\mathcal O}_Y(S_Y))
\,=\, (\varphi_*W)\otimes K_X\otimes {\mathcal O}_X(S_X)\, ,
$$
which will be denoted by $\varphi_*D$.
Since $D$ satisfies the Leibniz identity, it follows that $\varphi_*D$ also
satisfies the Leibniz identity. Hence $\varphi_*D$ is a logarithmic connection on
$\varphi_*D$ singular over $S_X$.
\end{proof}

\begin{theorem}\label{thm1}
Let $V_*$ be a parabolic vector bundle over $Y$ equipped with a connection $D$.
Then the logarithmic connection $\varphi_*D$ on $\varphi_* V$ is a connection
on the parabolic vector bundle $\varphi_*V_*$, where $V$ is the vector bundle
underlying $V_*$.
\end{theorem}

\begin{proof}
We need to show that the residues of $\varphi_*D$ have the required properties
(see Definition \ref{dlc}). Since it
is a purely local question, we may restrict $V_*$ to a neighborhood of a point
of $S_0\bigcup R$; note that the parabolic structure on $\varphi_*V_*$ over 
a point $x\, \in\, S_X$ was constructed using the decomposition in \eqref{g1},
so we can treat the points of $\varphi^{-1}(x)$ separately. Therefore,
we may assume that $\varphi^{-1}(x)$ has only one point. Also, locally, around a parabolic point,
a parabolic bundle can be expressed as a direct sum of parabolic line bundles.

Recall that the quasi-parabolic filtration of $\varphi_*V_*$ at $x$ is constructed by combining 
the filtrations of $F^k_\ell/F^k_{\ell+1}$, obtained from the quasi-parabolic filtration in 
\eqref{f5}, with the filtration in \eqref{f4}. Since $D$ is a connection on the parabolic bundle 
$V_*$, its residue at $y_k\,\in\, \varphi^{-1}(x)$ preserves the filtration in \eqref{f5}.

In view of the above observations, we may assume that $V$ is a line bundle, and $x$ is a point 
$S_X$ such that $\varphi^{-1}(x)$ is a point; the Riemann surfaces $X$ and $Y$ are not compact 
anymore because we do not need this assumption for local computations.

We need to show that the residue of the connection $\varphi_*D$ at the point $x$ preserves the 
quasi-parabolic filtration of $(\varphi_*V)_x$, and on each successive quotient of this 
filtration the residue acts as multiplication by the corresponding parabolic weight.

Let $Z$ be a Riemann surface and $y\, \in\, Z$ a point. Let $L$ be a holomorphic line bundle
on $Z$, and let $D_L$ be a logarithmic connection on $L$ singular over $y$ such that the
residue of $D_L$ at $y$ is $\tau\, \in\, \mathbb C$. For any integer $n$, consider the
holomorphic line bundle $L_n\, :=\, L\otimes {\mathcal O}_Z(n\cdot y)$. So $L$ and
$L_n$ are canonically identified over the complement $Z\setminus\{y\}$. The
connection $D_L$ on $L_n$ over $Z\setminus\{y\}$ extends to a logarithmic connection
on $L_n$. Indeed, for a holomorphic section $s$ of $L$ defined on an open neighborhood $U_y$ of $y$,
and a holomorphic function $f$, defined on $U_y$, such that $f(y)\,=\, 0$
and $df(y)\,\not=\, 0$, we have
$$
D_L(\frac{1}{f^n}s) \,=\, \frac{1}{f^n}\cdot D_L(s) -\frac{n\cdot df}{f} \frac{1}{f^n} s\, .
$$
This shows that the connection $D_L$ on $L_n$ is logarithmic. Moreover, the residue, at $y$,
of this logarithmic connection on $L_n$ is $\tau-n$.

Let $\psi\, :\, Z'\, \longrightarrow\, Z$ be a ramified covering of degree $\delta$ which is totally
ramified over $y$. Denote the point $\psi^{-1}(y)$ by $y'$. Then $\psi^*D_L$ is
a logarithmic connection on $\psi^*L$ singular over $y'$. The residue of $\psi^*D_L$ at
$y'$ is $\delta\cdot\tau$.

Using the above observations the theorem follows for the case of rank one. Note that
$c$ in \eqref{dk} arises because of the tensor product with the line bundle
${\mathcal O}_Y(-cy_k)$, and $b_k$ occurs in \eqref{dk} because the degree of ramification
of $\varphi$ at $y_k$ is $b_k$.

As noted before, it suffices to prove the rank one case. 
\end{proof}

\subsection{Ramified torus sub-bundle and connection}

Let $E_*$ be a parabolic vector bundle on $X$ equipped with a connection $D$.
The parabolic divisor for $E_*$ is $S\,\subset\, X$; the complement $X\setminus S$
is denoted by $U$. The logarithmic
connection $D$ on the underlying vector bundle $E$ induces a logarithmic connection
on the endomorphism bundle $\text{ad}(E)\,=\, \text{End}(E)$. Recall that the
connection $D$ is regular on $U$.

The above regular connection on $\text{ad}(E)\vert_U$ produces a connection on the
sub-fiber bundle $\text{Ad}(E)\vert_U\, \subset\, \text{ad}(E)\vert_U$. This connection on
$\text{Ad}(E)\vert_U$ will be denoted by $D^{\rm Ad}$. Let
$$
{\mathcal T}\, \subset\, \text{Ad}(E)\vert_U
$$
be a torus sub-bundle for the parabolic vector bundle $E_*$.

\begin{definition}\label{d-c-t}
We will say that $D$ \textit{preserves} $\mathcal T$ if the above connection
$D^{\rm Ad}$ on $\text{Ad}(E)\vert_U$ preserves the sub-bundle $\mathcal T$.
\end{definition}

Now consider the set-up of Theorem \ref{thm1}. Let $S_X\, \subset\, X$ be the parabolic
divisor for the parabolic vector bundle $\varphi_*V_*$; recall that it the union of
the ramification points for $\varphi$ and the image of the parabolic divisor for $V_*$.
The complement $X\setminus S_X$ will be denoted by $U$.

The parabolic vector bundle $\varphi_*V_*$ has a ramified torus sub-bundle
\begin{equation}\label{e5}
{\mathcal T}\, \subset\, \text{Ad}(\varphi_* V)\vert_U
\end{equation}
by Proposition \ref{prop2}.

\begin{lemma}\label{lem2}
The connection $\varphi_*D$ on $\varphi_* V$ preserves the ramified torus
sub-bundle $\mathcal T$ in \eqref{e5}.
\end{lemma}

\begin{proof}
It is a straight-forward consequence of the definition of the push-forward of a
connection by an \'etale morphism.
\end{proof}

\section{Characterization of direct images of parabolic bundles and connections}

\subsection{Characterization of direct images of parabolic bundles}\label{se5.1}

In Section \ref{se2} we constructed a parabolic structure on the direct image
of a vector bundle with parabolic structure, and in Proposition \ref{prop2} it
was shown that the parabolic bundle thus obtained is equipped with
a ramified torus sub-bundle, provided all the parabolic weights of the initial
parabolic bundle are rational numbers. We will now prove a converse of it.

Let $X$ be a connected smooth complex projective curve and $S\,:=\, \{x_1,\, \cdots,\, x_n\}
\, \subset\, X$ a
reduced effective divisor. Let $E_*$ be a parabolic vector bundle over $X$,
with parabolic structure over $S$ and underlying vector bundle $E$, such that
for every $x_i\, \in\, S$,
$$
E_{x_i}\,=\, F^i_1 \, 
\supsetneq\, F^i_2\, \supsetneq\, \cdots \, \supsetneq\, F^i_{\ell_i}\, \supsetneq\, 
F^i_{\ell_i+1}\,=\, 0\, .
$$
is the quasi-parabolic filtration and
$$
0\, \leq \, \lambda^i_1\, <\, \lambda^i_2\, 
<\, \cdots\, < \, \lambda^i_{\ell_i-1}\, < \, \lambda^i_{\ell_i} \, <\, 
\lambda^i_{\ell_i+1}\,=\, 1\, .
$$
are the corresponding parabolic weights.

All $\lambda^i_j$, $1\, \leq\, j\, \leq\,
\ell_i$, $1\,\leq\, i\, \leq\, n$, are assumed to be rational numbers.

\begin{proposition}\label{prop3}
Let ${\mathcal T}\, \subset\, {\rm Ad}(E)\vert_U$ be a ramified torus sub-bundle
for $E_*$, where $U\, :=\, X\setminus S$. Then there is a finite ramified covering
$$
\varphi\ :\, Y\, \longrightarrow\, X
$$
and a parabolic vector bundle $V_*$ on $Y$ with rational parabolic weights, such that
the parabolic vector bundle $\varphi_*V_*$ is isomorphic to $E_*$. Moreover the
isomorphism between $\varphi_*V_*$ and $E_*$ can be so chosen that it takes
$\mathcal T$ to the ramified torus sub-bundle for $\varphi_*V_*$ given by
Proposition \ref{prop2}.
\end{proposition}

\begin{proof}
There is a (ramified) Galois covering
$$
\phi\, :\, \widetilde{X}\, \longrightarrow\, X
$$
and a $\text{Gal}(\phi)$--equivariant vector bundle $\widetilde E$ on $\widetilde X$,
such that
\begin{itemize}
\item the parabolic vector bundle $E_*$ corresponds to the
$\text{Gal}(\phi)$--equivariant vector bundle $\widetilde E$, and

\item the torus sub-bundle $\phi^*{\mathcal T} \, \subset\,
\text{Ad}(\widetilde{E})\vert_{\phi^{-1}(U)}$ extends to a torus sub-bundle
\begin{equation}\label{te}
\widetilde{\mathcal T}\, \subset\, \text{Ad}(\widetilde{E})
\end{equation}
over entire $\widetilde X$.
\end{itemize}
For notational convenience, the Galois group $\text{Gal}(\phi)$ will be denoted
by $G$.

The Lie algebra bundle on $\widetilde X$ for the group scheme $\widetilde{\mathcal T}$
will be denoted by $\mathcal A$. So $\mathcal A$ is also a sub-algebra bundle of the
associative algebra bundle on $\widetilde X$
\begin{equation}\label{cA}
{\mathcal A}\,=\, \text{Lie}(\widetilde{\mathcal T})\,\subset\, \text{Lie}(\text{Ad}(\widetilde{E}))
\,=\, {\rm End}(\widetilde{E})\, ;
\end{equation}
the Lie algebra structure is given by the commutator. The total space of the
associative algebra bundle $\mathcal A$ will also be denoted by $\mathcal A$. Let
\begin{equation}\label{z0}
Z_0\, \subset\, \mathcal{A}
\end{equation}
be the sub scheme defined by the equation $z^2-z\,=\, 0$. So $Z_0$ is the locus of all
idempotent elements in the associative algebra bundle $\mathcal A$. It is
straight-forward to check that $Z_0$ is an \'etale cover of $\widetilde{X}$. This
$Z_0$ is smooth but not connected; note that the zero section of $\mathcal A$ is a
connected component of $Z_0$.

Let
$$
Z\, \subset\, Z_0
$$
be the locus of nonzero elements that are not nontrivial sum of other elements of $Z_0$.
This $Z$ is in fact a union of some connected components of $Z_0$. Let
\begin{equation}\label{psi}
\psi\, :\, Z\, \longrightarrow\, \widetilde{X}
\end{equation}
be the restriction of the natural projection of $\mathcal A$ to $\widetilde X$.

Note that for any $x\, \in\, \widetilde{X}$, we have the direct sum decomposition
\begin{equation}\label{deE}
\widetilde{E}_x\,=\, \bigoplus_{z\in\, \psi^{-1}(x)} z(\widetilde{E}_x)\, ,
\end{equation}
where $\psi$ is the projection in \eqref{psi};
recall that $z$ is an idempotent endomorphism of $\widetilde{E}_x$. From \eqref{deE}
it follows that the pulled back vector bundle $\psi^*\widetilde{E}$ on $Z$ has a
tautological sub-bundle
\begin{equation}\label{W}
W\, \subset\, \psi^*\widetilde{E}\, \longrightarrow\, Z
\end{equation}
whose fiber over any $z\, \in\, Z$ is the image $z(\widetilde{E}_{\psi(z)})\,\subset\,
\widetilde{E}_{\psi(z)}\,=\, (\psi^*\widetilde{E})_z$. To describe $W$ in another
way, note that since $Z$ is contained in the total space of $\text{End}(\widetilde{E})$,
there is a tautological homomorphism
$$
\psi^*\widetilde{E}\, \longrightarrow\, \psi^*\widetilde{E}\, .
$$
The sub-bundle $W$ in \eqref{W} is the image of this tautological homomorphism.

The sub-bundle $W$ of $\psi^*\widetilde{E}$
is a direct summand. In fact, it has a tautological complement $W^c$ whose
fiber over any $z\, \in\, Z$ is
$$
W^c_z\,=\, \text{kernel}(z)\,=\, 
\bigoplus_{y\in\, \psi^{-1}(\psi(z))\setminus\{z\}} y(\widetilde{E}_{\psi(z)})
\, \subset\, (\psi^*\widetilde{E})_z\, .
$$
It should be clarified that, in general, the decomposition in \eqref{deE} does not
produce any decomposition of $W^c$, because the direct summands in the fiber
$W^c_z$ may get interchanged as $z$ runs over a loop in $Z$.

Consider the action of $\text{Gal}(\phi)\,=\, G$ on $\text{End}(\widetilde{E})$ induced by the 
action of $G$ on $\widetilde E$. It can be shown that both $\widetilde{\mathcal T}$ and 
$\mathcal A$ (see \eqref{cA}) are preserved by this action of $G$ on 
$\text{End}(\widetilde{E})$. Indeed, this follows immediately from the fact that 
$\widetilde{\mathcal T}\vert_{\phi^{-1}(U)}$ is pulled back from $X$. The action of $G$ on 
$\mathcal A$ evidently preserves $Z_0$ in \eqref{z0}. The action of $G$ on $Z_0$ clearly 
preserves $Z$ in \eqref{psi}. The map $\psi$ in \eqref{psi} is evidently $G$--equivariant.

The action of $G$ on $\widetilde E$ pulls back to an action of
$G$ on $\psi^*{\widetilde E}$ such that the projection map from $\psi^*{\widetilde E}$ to
$Z$ intertwines the actions of $G$. The sub-bundle $W$ in \eqref{W} is clearly
preserved by this action of $G$ on $\psi^*{\widetilde E}$. 
Consequently, the vector bundle $W$ on $Z$ is $G$--equivariant.

The quotient $Z/G$ will be denoted by $Y$. It was noted above $\psi$ is
$G$--equivariant. Therefore, the map $\psi$ descends to a map
$$
\varphi\, :\, Y\, \longrightarrow\, X\, .
$$
More precisely, we have a commutative diagram of morphisms
$$
\begin{matrix}
Z & \stackrel{\psi}{\longrightarrow} & \widetilde{X}\\
~\Big\downarrow q && ~ \Big\downarrow\phi\\
Y & \stackrel{\varphi}{\longrightarrow} & X
\end{matrix}
$$
where $q$ is the quotient map to $Z/G$. 

Let $V_*$ be the parabolic vector bundle on $Y$ associated to the above
$G$--equivariant vector bundle $W$ on $Z$ \cite{Bi1}, \cite{Bo1}.
Note that the parabolic points of $V_*$ are contained in $\varphi^{-1}(S)$.

Let any point $x\, \in\, U\,=\, X\setminus S$. The fibers of both $\varphi_*V_*$ and
$E_*$ over $x$ are identified with
$$
\bigoplus_{z\in\, \psi^{-1}(y)} z(\widetilde{E}_y)\, ,
$$
where $y$ is any point of $\phi^{-1}(x)$; note that for different choices
of $y$, the corresponding direct sums get identified using the action of $G$.
In other words, on $U$, the vector bundles $E_*$ and $\varphi_*V_*$
are canonically identified. This isomorphism clearly takes the torus sub-bundle
${\mathcal T}\, \subset\, \text{Ad}(E)\vert_U$ to the torus over sub-bundle
of $\text{Ad}(\varphi_*V_*)\vert_U$ given by Proposition \ref{prop2}.

Using the above constructions, it is straight-forward to check that this isomorphism 
between $E_*$ and $\varphi_*V_*$ over $U$ extends to an isomorphism between the 
parabolic vector bundles $E_*$ and $\varphi_*V_*$ over $X$. It was already noted that
this isomorphism takes ${\mathcal T}$ to the torus sub-bundle
of $\text{Ad}(\varphi_*V_*)\vert_U$ given by Proposition \ref{prop2}.
\end{proof}

\begin{remark}\label{rem1}
Let $\phi^1\, :\, Z\, \longrightarrow\, X$ be a ramified covering map, where $X$
is a connected smooth complex projective curve. Let $V^1_*$ be a parabolic vector
bundle on $Z$ with rational parabolic weights. We have the parabolic vector
bundle $\phi^1_* V^1_*$ on $X$ constructed in Section \ref{se2}. Since all the
parabolic weights of $V^1_*$ are rational numbers, the parabolic
weights of $\phi^1_* V^1_*$ are rational too. Let
$$
S\, \subset\, X
$$
be the parabolic divisor for $\phi^1_* V^1_*$.

The ramified
torus sub-bundle for $\phi^1_* V^1_*$ constructed in Proposition \ref{prop2}
will be denoted by ${\mathcal T}^1$.

Consider the pair $(\phi^1_* V^1_*,\, {\mathcal T}^1)$. Using it, Proposition \ref{prop3}
produces a ramified covering map
$$
\varphi\,:\, Y\, \longrightarrow\, X
$$
and a parabolic bundle $V_*$ on $Y$ (see the proof of Proposition \ref{prop3}).
Consider the two coverings of $X\setminus S$
$$
\phi^1_0\, :\, (\phi^1)^{-1}(X\setminus S)\, \longrightarrow\, X\setminus S\ \
\text{ and }\ \ \varphi_0\, :\, \varphi^{-1}(X\setminus S)\,
\longrightarrow\, X\setminus S\, ,
$$
where $\phi^1_0$ (respectively, $\varphi_0$) is the restriction of 
$\phi^1$ (respectively, $\varphi$) to the Zariski open subset $(\phi^1)^{-1}(X\setminus S)$
(respectively, $\varphi^{-1}(X\setminus S)$) of $Z$ (respectively, $Y$). It is
straight-forward to check that these two coverings of $X\setminus S$ are canonically
identified. In other words, we have commutative diagram of maps
$$
\begin{matrix}
(\phi^1)^{-1}(X\setminus S) & \stackrel{\delta}{\longrightarrow} &
\varphi^{-1}(X\setminus S) \\
~\Big\downarrow \phi^1_0 && ~ \Big\downarrow\varphi_0\\
X\setminus S & \stackrel{\rm Id}{\longrightarrow} & X\setminus S
\end{matrix}
$$
where $\delta$ is an isomorphism. Since $(\phi^1)^{-1}(X\setminus S)$
and $\varphi^{-1}(X\setminus S)$ are Zariski open dense in $Z$ and $Y$ respectively,
the above map $\delta$ extends to an isomorphism
$$
\widehat{\delta}\, :\, Z\, \longrightarrow\, Y
$$
such that the diagram
\begin{equation}\label{wdi}
\begin{matrix}
Z & \stackrel{\widehat\delta}{\longrightarrow} & Y \\
~\Big\downarrow \phi^1 && ~ \Big\downarrow\varphi\\
X & \stackrel{\rm Id}{\longrightarrow} & X
\end{matrix}
\end{equation}
is commutative.

The pulled back vector bundle $\delta^* (V_*\vert_{\varphi^{-1}(X\setminus S)})$
is identified with $V^1_*\vert_{(\phi^1){-1}(X\setminus S)}$. Note that the
parabolic divisor for the parabolic vector bundle $V_*$ (respectively, $V^1_*$)
is contained in $\varphi^{-1}(S)$ (respectively, $(\phi^1)^{-1}(S)$). It can checked
that this identification extends to an isomorphism between the parabolic vector bundles
${\widehat\delta}^*V_*$ and $V^1_*$ over $Z$, where $\widehat\delta$ is the isomorphism
in \eqref{wdi}.
\end{remark}

Combining the construction in Section \ref{se2} with Proposition \ref{prop2},
Proposition \ref{prop3} and Remark \ref{rem1}, we have the following theorem.

\begin{theorem}\label{thm2}
Let $E_*$ be a parabolic vector bundle on a connected smooth complex projective curve $X$
with parabolic divisor $S$ and rational parabolic weights. Then there is a natural
equivalence between the following two classes:
\begin{enumerate}
\item Triples of the form $(Y,\, \varphi,\, V_*)$, where $\varphi\, :\, Y\, \longrightarrow
\, X$ is a ramified covering map, and $V_*$ is a parabolic vector bundle on $Y$, such that
$\varphi_*V_*\, =\, E_*$.

\item Ramified torus bundles for $E_*$.
\end{enumerate}
\end{theorem}

\subsection{Characterization of direct images of connections}

As in Section \ref{se5.1}, $X$ is a connected smooth complex projective curve and 
$S\, \subset\, X$ a reduced effective divisor. Let 
$E_*$ be a parabolic vector bundle over $X$, with parabolic structure over $S$ and 
underlying vector bundle $E$, such that all the parabolic weights are rational 
numbers. Let
\begin{equation}\label{rt}
{\mathcal T}\, \subset\, {\rm Ad}(E)\vert_U
\end{equation}
be a ramified torus 
sub-bundle for $E_*$, where $U\, :=\, X\setminus S$. Proposition
\ref{prop3} says that there is a finite ramified covering
$$
\varphi\ :\, Y\, \longrightarrow\, X
$$
and a parabolic vector bundle $V_*$ on $Y$ with rational parabolic weights, such that
\begin{itemize}
\item the parabolic vector bundle $\varphi_*V_*$ is isomorphic to $E_*$, and

\item the
isomorphism between $\varphi_*V_*$ and $E_*$ can be so chosen that it takes
$\mathcal T$ to the ramified torus sub-bundle of $\varphi_*V_*$ given by
Proposition \ref{prop2}.
\end{itemize}

\begin{proposition}\label{prop4}
Let $D$ be a connection on the parabolic vector bundle $E_*$ such that $D$
preserves the ramified torus sub-bundle $\mathcal T$ in \eqref{rt}. Then
$D$ produces a connection $D'$ on the above parabolic vector bundle $V_*$ such that
the connection $\varphi_* D'$ on $\varphi_*V_*$ (see Theorem \ref{thm1}
for $\varphi_* D'$) coincides with the connection $D$ on $E_*$.
\end{proposition}

\begin{proof}
Since $E_*$ corresponds to the $\text{Gal}(\phi)$--equivariant vector bundle
$\widetilde{E}\, \longrightarrow\, \widetilde{X}$ in \eqref{te}, the connection
$D$ on $E_*$ produces a $\text{Gal}(\phi)$--invariant connection on $\widetilde{E}$.
This connection on $\widetilde{E}$ given by $D$ will be denoted by $\widetilde D$.
On $\phi^{-1}(U)$, the vector bundle $\widetilde{E}\vert_{\phi^{-1}(U)}$
is identified with the pullback $\phi^*(E\vert_U)$. Recall that this isomorphism
over $\phi^{-1}(U)$ takes $\phi^*{\mathcal T}$ to $\widetilde{\mathcal T}$
in \eqref{te}.

Let $\widetilde{D}'$ be the connection on the vector bundle $\text{End}(\widetilde{E})$
induced by the above connection $\widetilde D$ on $\widetilde{E}$.
Since $D$ preserves the ramified torus sub-bundle $\mathcal T$ in \eqref{rt}, it follows
that $\widetilde{D}'$ preserves $\widetilde{\mathcal T}$ over $\phi^{-1}(U)$. As
$\phi^{-1}(U)$ is Zariski open dense in $\widetilde X$, this implies that
$\widetilde{D}'$ preserves $\widetilde{\mathcal T}$ over $\widetilde X$. Therefore,
the connection $\widetilde{D}'$ preserves $\mathcal A$ in \eqref{cA}.

Since $\widetilde{D}'$ preserves $\mathcal A$, we conclude that $\widetilde{D}'$
preserves the sub-scheme $Z_0$ in \eqref{z0}, meaning for every $z\, \in\, Z_0$, the
horizontal subspace of $T_z\text{End}(\widetilde{E})$ for the connection
$\widetilde{D}'$ actually lies inside the subspace
$$
T_z Z_0\, \subset\, T_z\text{End}(\widetilde{E})\, .
$$
Note that this implies that the above
horizontal subspace of $T_z\text{End}(\widetilde{E})$ coincides with
$T_z Z_0$, because $Z$ is an \'etale cover of $\widetilde X$. We now conclude that $Z$
in \eqref{psi} is also preserved by $\widetilde{D}'$, because $Z$ is a union of some
connected components of $Z_0$

Moreover, the connection $\psi^*\widetilde{D}$ on $\psi^*\widetilde{E}$
preserves the tautological sub-bundle $W$ in \eqref{W}. Indeed, this follows
from the fact that the connection $\widetilde{D}'$ preserves $\mathcal A$.
Let $D^W$ denote the connection on $W$ given by $\psi^*\widetilde{D}$.

Recall that the connection $\widetilde D$ is preserved by the action of the Galois
group $\text{Gal}(\phi)$ on $\widetilde E$. Hence the connection $\widetilde{D}'$
on $\text{End}(\widetilde{E})$ is also $\text{Gal}(\phi)$--invariant. Consequently,
the above connection $D^W$ on $W$ in preserved by the action $\text{Gal}(\phi)$
on $W$. This implies that $D^W$ produces a connection on the parabolic vector bundle
$V_*$ on $Z/\text{Gal}(\phi)\,=\, Y$.

The above connection on $V_*$ given by $D^W$ produces a connection on the direct 
image $\varphi_*V_*$ (see Theorem \ref{thm1}). This connection on 
$\varphi_*V_*\,=\, E_*$ (see Proposition \ref{prop3}) will be 
denoted by $\widehat D$.

The identification between $\varphi_*V_*$ and $E_*$ takes the restriction of the connection
$D$ to $E\vert_U$ to the restriction of $\widehat D$ to $(\varphi_*V_*)\vert_U$.
Therefore, the connection $D$ on $E_*$ coincides with the connection
$\widehat D$ on $\varphi_*V_*$.
\end{proof}

\begin{remark}\label{rem2}
Consider the set-up of Remark \ref{rem1}. Let $D^1$ be a connection on the
parabolic vector bundle $V^1_*$ over $Z$. Using Theorem
\ref{thm1}, this $D^1$ produces a connection
on the parabolic vector bundle $\phi^1_*V^1_*$ over $X$; the
connection on $\phi^1_*V^1_*$ given by $D^1$ will be denoted by
$\phi^1_* D^1$. From Lemma \ref{lem2} we know that this connection $\phi^1_* D^1$
preserves the ramified torus sub-bundle ${\mathcal T}^1$ in
Remark \ref{rem1} for $\phi^1_*V^1_*$.

Now from Proposition \ref{prop4} and Remark \ref{rem1} we know that $\phi^1_* D^1$
produces a connection on $V^1_*$; this connection on $V^1_*$ will be denoted by $D^2$.
The two connections $D^1$ and $D^2$ on the parabolic vector bundle $V^1_*$ coincide
over the dense open subset $(\phi^1)^{-1}(X\setminus S)$. This implies that
$D^1$ and $D^2$ coincide over $Z$.
\end{remark}

Combining Lemma \ref{lem2}, Proposition \ref{prop4} and Remark \ref{rem2}, we have
the following theorem.

\begin{theorem}\label{thm3}
The equivalence in Theorem \ref{thm2} takes a connection on the parabolic
vector bundle $V_*$ on $Y$ to a connection on $E_*$ that preserves the
ramified torus sub-bundle for $E_*$ corresponding to $(Y,\, \varphi,\, V_*)$. Conversely,
a connection on $E_*$ preserving a ramified torus sub-bundle $\mathcal T$ for $E_*$
is taken to a connection on the parabolic vector bundle $V_*$ on $Y$, where
$(Y,\, \varphi,\, V_*)$ corresponds to $\mathcal T$.
\end{theorem}



\begin{thebibliography}{AAAA}

\bibitem[At]{At} M. F. Atiyah, Complex analytic connections in fibre
bundles, \textit{Trans. Amer. Math. Soc.} \textbf{85} (1957), 181--207.

\bibitem[AB]{AB} R. Auffarth and I. Biswas, Direct images of line bundles, 
\textit{Jour. Pure Appl. Alg.} {\bf 222} (2018), 1189--1202.

\bibitem[Bi1]{Bi1} I. Biswas, Parabolic bundles as
orbifold bundles, \textit{Duke Math. Jour.}
\textbf{88} (1997), 305--325.

\bibitem[Bi2]{Bi2} I. Biswas, Parabolic ample bundles,
\textit{Math. Ann.} \textbf{307} (1997), 511--529.

\bibitem[BCW]{BC} I. Biswas and C. Gangopadhyay, M. L. Wong, Direct images of vector
bundles and connections, {\it Beit. Alg. Geom.}, https://doi.org/10.1007/s13366-018-0410-x.

\bibitem[BL]{BL} I. Biswas and M. Logares, Connection on parabolic vector bundles over 
curves, {\it Inter. Jour. Math.} {\bf 22} (2011), 593--602.

\bibitem[Bo]{Bo1} N. Borne, Fibr\'es paraboliques et champ des racines,
\textit{Int. Math. Res. Not.} (2007), no. 16, Art. ID rnm049.

\bibitem[De]{De} P. Deligne, {\it \'Equations diff\'erentielles \`a points
singuliers r\'eguliers}, Lecture Notes in Mathematics, Vol. 163, Springer-Verlag,
Berlin-New York, 1970.

\bibitem[DePa]{DP} A. Deopurkar and A. Patel, Vector bundles and finite covers,
arXiv:1608.01711 [math.AG].

\bibitem[DG]{DG} R. Donagi and D. Gaitsgory, The gerbe of Higgs bundles,
{\it Transform. Gr.} {\bf 7} (2002), 109--153. 

\bibitem[DoPa]{DoPa} R. Donagi and T. Pantev, Langlands duality for Hitchin systems,
{\it Invent. Math.} {\bf 189} (2012), 653--735.

\bibitem[MY]{MY} M. Maruyama and K. Yokogawa, Moduli of
parabolic stable sheaves, \textit{Math. Ann.} \textbf{293}
(1992), 77--99.

\bibitem[MS]{MS} V. B. Mehta and C. S. Seshadri, Moduli of vector bundles on curves
with parabolic structure, {\it Math. Ann.} \textbf{248} (1980), 205--239.

\bibitem[Na]{Na} M.~Namba, {\it Branched coverings and algebraic
functions}, Pitman Research Notes in Mathematics Series, 161,
Longman Scientific $\&$ Technical, Harlow; John Wiley $\&$ Sons,
Inc., New York, 1987.

\bibitem[Oh]{Oh} M. Ohtsuki, A residue formula for Chern classes associated with
logarithmic connections, {\it Tokyo Jour. Math.} {\bf 5} (1982), 13--21.

\bibitem[Yo]{Yo} K. Yokogawa, Infinitesimal deformations of
parabolic Higgs sheaves, \textit{Int. Jour. Math.}
\textbf{6} (1995), 125--148.

\end{thebibliography}
\end{document}